\documentclass[11pt]{article}

\usepackage[a4paper, total={6.4in, 8.5in}]{geometry}
\usepackage{amsmath,amssymb,latexsym,amsfonts,amsthm}
\usepackage{stmaryrd}
\usepackage[new]{old-arrows}
\usepackage{shapepar,bm}
\usepackage[arrow,matrix,ps,color,line]{xy}
\usepackage{amsfonts,amssymb,amscd,hyperref,amsthm}
\usepackage{geometry}
\usepackage{makeidx}
\usepackage{eulervm}
\usepackage{mathrsfs}

\usepackage{graphicx}
\usepackage{graphics}

\usepackage{epstopdf}

\hypersetup{
    colorlinks = true,
    linkcolor = blue,
    anchorcolor = red,
   citecolor = blue,
    filecolor = red,
    urlcolor = blue}

\newcommand{\mathsym}[1]{{}}

\DeclareMathOperator{\Spec}{Spec}

\renewcommand{\H}{\ensuremath{\mathrm{S}}}
\renewcommand{\k}{\mathbf{k}} 
\newcommand{\C}{\ensuremath{\mathcal{C}}}
\newcommand{\proj}{\ensuremath{\mathbb{P}}}
\renewcommand{\l}{\ensuremath{\ell}}
\newcommand{\MS}{\ensuremath{\mathcal{M}_{g,1}^{\H}}} 
\renewcommand{\dim}{\ensuremath{\mathrm{dim}}}%

\newcommand{\LO}{\mathcal{O}} 
\renewcommand{\k}{\mathbf{k}} 
\newcommand{\bfx}{\mathbf{x}} 

\def\balpha{{\bm\alpha}}

\usepackage[usenames]{color}
\definecolor{MyLightMagenta}{cmyk}{0.1,0.8,0,0.1}
\definecolor{MyDarkBlue}{rgb}{0.1,0,0.3}
\definecolor{redb}{rgb}{0.76, 0.13, 0.28}
\definecolor{redb}{rgb}{0.76, 0.13, 0.28}
\definecolor{myblue}{rgb}{0.0, 0.53, 0.74}
\definecolor{tangerine}{rgb}{1.0, 0.66, 0.07}
\hypersetup{colorlinks,linkcolor={myblue},citecolor={myblue},urlcolor={myblue}} 

\makeindex

\theoremstyle{plain}
\newtheorem{thm}{Theorem}[section]

\newtheorem{exam}{Example}[section]

\newtheorem{rmk}[thm]{Remark}

\newtheorem{lem}[thm]{Lemma}
\newtheorem{cor}[thm]{Corollary}

\newtheorem*{mainA}{Main Theorem}

\def\no@breaks#1{{\def\\{ \ignorespaces}#1}}    

  \makeatletter
\def\cleardoublepage{\clearpage\if@twoside \ifodd\c@page\else
\hbox{} \thispagestyle{empty}
\newpage
\if@twocolumn\hbox{}\newpage\fi\fi\fi} \makeatother

\usepackage{eso-pic,graphicx}
\makeatletter
\newcommand\BackgroundPicture[2]{%
  \setlength{\unitlength}{1pt}%
  default \put(0,\strip@pt\paperheight){%
  \parbox[t][\paperheight]{\paperwidth}{%
    \vfill
     \centering \includegraphics[angle=#2, width=15cm, height=15cm,  bb=0 0 150 150]{#1}
    \vfill
}}} %
\makeatother

\begin{document}

\date{}
\title{Local Complete Intersections and Weierstrass Points}

\author{Andr\'e Contiero \& Sarah Mazzini
\thanks{During the preparation of this paper the first author was partially support by ICTP-INdAM Research in Pairs Programme, Trieste, Italy and by Funda\c c\~ao de Amparo \`a Pesquisa do Estado
de Minas Gerais (FAPEMIG), Brazil, grant no. APQ-00798-18. The second author was financed in part by the 
Coordena\c c\~ao de Aperfei\c coamento de Pessoal de N\'ivel Superior - Brasil (CAPES) - Finance Code 001.
\smallskip 
\newline  ${}$
\newline ${}$ \,\,\,\,\,\, {\bf 2020 MSC:} 14H10,\,14H55, 14B07, 14M10.}}

\maketitle

\begin{abstract}
\noindent This work presents a simple proof that the moduli space of complete integral Gorenstein curves with a prescribed symmetric Weierstrass semigroup becomes a weighted projective space, even for fields of positive characteristic, when the associated monomial curve is a local complete intersection.
\end{abstract}


\section{Introduction}\label{Intro}

Given a numerical semigroup $\H\subset\mathbb{N}$ of genus $g\geq 1$, minimally 
generated by $a_1,\dots,a_r$, let $\MS$ be the
moduli space parameterizing smooth pointed curves defined 
over an algebraically closed field $\k$ (or compact Riemann surfaces when $\k=\mathbb{C}$),
whose Weierstrass semigroup at the marked point is $\H$. It is
well known that $\MS$ can be empty depending on $\H$, but when it is non-empty, a major and very classical problem is to describe the moduli space $\MS$ and its compactification.

By allowing singularities, any numerical semigroup $\H$ can be realized as the Weierstrass semigroup 
of a projetivization of the affine monomial curve 
$$\C_{\H}:=\{(t^{a_1},\dots,t^{a_r})\,;\,t\in\k\}\subset\mathbb{A}^r.$$ Herzog \cite{Her70} showed 
that the ideal of $\C_{\H}$ can be generated by polynomials 
in $\k[X_{a_1},\dots,X_{a_r}]$ which are differences of two monomials with the same weighted degree, namely  
\begin{equation}\label{polymain}
    G^{(0)}_{d_j}:=X_{a_1}^{\alpha_{1,j}}\cdots X_{a_r}^{\alpha_{r,j}}-X_{a_1}^{\beta_{1,j}}\cdots X_{a_r}^{\beta_{r,j}},
\end{equation}
where $\alpha_{i,j}\cdot \beta_{i,j}=0$ and $\sum a_i\alpha_{i,j}=\sum a_i\beta_{i,j}$ for $1\leq i\leq r$ and  $1\leq j\leq m$.

The purpose of this paper is to establish the following result:

\begin{mainA}\label{thmA}
If $\H$ is such that the affine monomial curve $\C_{\H}=\Spec\k[\H]$ is a local complete intersection and $\mathrm{char}(\k)=0$ or a prime not dividing any exponent $\alpha_{i,j}$ and $\beta_{i,j}$ of the defining equations of $\C_{\H}$, then a compactification
$$\overline{\MS}=\proj(\mathbf{T}^{1,-}),$$ 
it is constructed and the closure is compounded by integral Gorenstein curves with a smooth point whose Weierstrass semigroup is $\H$. 
\end{mainA}

The vector space $\mathbf{T}^{1,-}$ 
stands for the negatively graded part of the first module of the cotangent complex associated to the semigroup algebra $\k[\H]=\oplus_{n\in\H}\k t^n\,$, $$\mathbf{T}^{1}(\k[\H])=\mathbf{T}^{1,-}(\k[\H])\oplus \mathbf{T}^{1,+}(\k[\H]).$$ 

We recall that a numerical semigroup $\H$ is a complete intersection if the affine monomial curve $\C_{\H}$ is a complete intersection in
$\mathbb{A}^r$, where $r$ is the embedding dimension of $\C_{\H}$, i.e. the smallest number of elements required to generate $\H$. Equivalently, the semigroup algebra $\k[\H]$ is a complete intersection when we consider it as the quotient of $\k[X_{a_1},\dots,X_{a_r}]$ by the kernel $\mathbf{I}$ of the surjective map
$$
\begin{array}{rcl}
\k[X_{a_1},\dots,X_{a_r}] & \longrightarrow & \k[\H] \\
X_{a_i} & \longmapsto & t^{a_i},
\end{array}
$$ and $\mathbf{I}$ is the defining ideal of $\C_{\H}\subset\mathbb{A}^r$.

The affine monomial curve $\C_{\H}$ has a unique unibranch singular point at the origin $\mathbf{0}$, with singularity degree $g=g(\H)$. Therefore $\C_{\H}$, or even its closure in a suitable (weighted) projective space, is a local complete intersection if and only if the local ring of its unique singularity is a complete intersection. Since $\C_{\H}$ is affine and a locally complete intersection, a minimal free resolution of the local ring singularity lifts to a minimal free resolution of the semigroup algebra $\k[\H]$, and hence $\C_{\H}$ is a global complete intersection in $\mathbb{A}^r$.

If $\k[\H]$ is a complete intersection, then there are no obstructions to formally deform $\C_{\H}$ 
in characteristic zero, the second cohomology module of the cotangent complex associated to 
$\C_{\H}$ is null, $\mathbf{T}^{2}_{\C_{\H}}=0$, as shown in \cite{LS}. Hence, we can conclude that 
$\MS$ is smooth. Furthermore, the base space $\mathcal{T}^{-}$ of the miniversal deformation in 
negative degrees is an affine space $\mathbb{A}^{N}$. Therefore, we can deduce that a closure of 
$\MS$ is also a projective space, whenever we apply Pinkham's construction of $\MS$ for smooth 
fibers $\mathcal{X}^{-}\rightarrow\mathcal{T}^{-}$ of the miniversal deformation (see Section 
\ref{secHauser} for more details). The advantage of our techniques is that the proof of the Main Theorem presented here is rather explicit and simple, and it also works for fields of 
positive characteristic and describes the curves that compound the boundary.

A rather simple proof of the Main Theorem in characteristic zero can be obtained as follows.
\begin{proof}
The dimension of $\MS$ is at least
$2g-1-\dim\mathbf{T}^{1,+}$, c.f. \cite[Theorem 2.4]{CFSV}.
Since the monomial curve in $\proj^{g-1}$ associated to $\C_{\H}$ is a local complete intersection and $\mathrm{char}(\k)=0$, we are able to show 
that $\dim\mathbf{T}^{1,-}=2g-\dim\mathbf{T}^{1,+}$, meaning that the Tjurina number of a complete intersection singularity is $2g$. At this point we just have to apply the results due to Stoehr and Contiero--Stoehr \cite{St93,CS} 
assuring that $\overline{\MS}$ is a closed subset of $\proj(\mathbf{T}^{1,-})$.   
\end{proof}

The way we prove the Main Theorem for also fields of positive characteristic is to apply a variant of Hauser's algorithm (see \cite{Hau83,Hau85} and \cite{Stev13}) by deforming the affine monomial curve $\C_{\H}\subset\mathbb{A}^r$ instead of the associated canonical Gorenstein monomial curve in $\proj^{g-1}$, as required by Stoehr's original construction (\cite{St93}). The first step is to take the unfold of the $r-1$ defining equations of the ideal of $\C_{\H}$. Next, since $\C_{\H}$ is a complete intersection, we can show that no relations between the unfolded coefficients arise from syzygies, with the exception of $\frac{1}{2}r(r+1)$ normalizations to zero. This is where the condition on the characteristic of the ground field appears. Hence, the closure of the moduli space ${\MS}$ is $\proj(V)$, where $V$ is the $\k$-vector space spanned by the normalized unfolded coefficients. Finally, we just need to note that V is in bijection with $\mathbf{T}^{1,-}$, c.f. \cite[Appendix]{St93}.

We obtain the following two naive and immediate consequences of the above Main Theorem, provided that $\mathrm{char}(\k)=0$ or a prime not dividing any exponent $\alpha_{i,j}$ and $\beta_{i,j}$ of the defining equations of $\C_{\H}$.

\begin{cor}[Schlessinger \cite{SchPhD} and Pinkham \cite{Pi74}]
A complete intersection numerical semigroup is 
realized as a Weierstrass semigroup of a smooth curve.
\end{cor} 

\begin{cor}
If $\C_{\H}$ is a local complete intersection, then the associated affine
monomial curve can be negatively smoothed without any obstruction.
\end{cor}

In general, it is very difficult to describe a compactification of $\MS$ and the curves that make up its boundary. The authors are aware of two main approaches to considering geometric features of a closure of $\MS$ and properties of curves on its boundary. In the following two subsections, we cite some results concerning these two approaches. There are many high-standing works that are not cited here, most of which are referenced in the works cited below.

\subsection{$\MS$ coming from versal deformation}

The general theory of versal deformations of singularities dates back to the 1960s and 1970s, with the remarkable works of Schlessinger \cite{SchPhD} and Artin \cite{Ar76}. The connection between the spaces $\MS$ and the miniversal deformation in negative degrees was made by Pinkham in his PhD thesis  \cite{Pi74}, using an affine monomial curve associated with the semigroup $\H$. We shall briefly describe this connection in Subsection \ref{secHauser} below, as it is one of the main techniques used in this paper.

Some works have explored the study of $\MS$ using versal deformations. As Pinkham's paper \cite{Pi74} shows, the miniversal deformation offers a way to construct a compactification of $\MS$. The resulting closure of $\MS$ is totally described just for a few families of
semigroups, as we note below.

In \cite{St93}, Stoehr presents a rather explicit way to construct a compactification of $\MS$ as a variant of Hauser's algorithm, when $\H$ is assumed to be a suitable symmetric semigroup. Stoehr's construction relies on the unfold of the defining equations of the canonically embedded projective monomial curve associated to $\C_{S}$, extending Petri's analysis of the canonical ideal and then explo\-ring appropriate syzygies coming from the defining equations of $\C_{\H}$. It is obtained a compactification of $\MS$ as a closed subset of a weighted projective space by allowing irreducible Gorenstein curves at the boundary. Later on, Contiero-Stoher \cite{CS} and Contiero-Fontes \cite{CF} extend Stoehr's construction to all symmetric semigroups, making it totally implementable as well. In Section \ref{secHauser} below, we briefly recall this construction. We also refer to \cite{Maz}, where the second author presents some algorithms to compute the defining equations of $\C_\H$ and their unfolding, the defining equations of $\overline{\MS}$, and the equivariant tangent space $\mathbf{T}^{1}_{\C_\H}$ of the versal deformation space, whenever $\H$ is symmetric.

Nakano \cite{Nak} computed $\MS$ using Pinkham's approach by computationally determining the base space of the miniversal deformation of the monomial curve $\C_\H$ in negative degrees, for $g\leq 5$. He shows that for $g\leq 5$ the base space is an irreducible rational variety, except in one case: the semigroup $<4,6,11,13>$ when it has the structure of a projective quasi-cone over $\proj^1 \times \proj^3$. In this case, the base space is also irreducible, but in negative degrees it contains two components, one smooth and the other containing a curve with a double point (see \cite[Remark 2.9]{CFSV}).

In a recent paper \cite{Stev22}, Stevens extends the results of Nakano \cite{Nak} and explicitly computes the defining equations of the moduli space $\MS$ for many cases of genus at most seven and determines the dimension for all semigroups of genus not greater than seven. Stevens uses Hauser's algorithm in most cases, but in one case, he uses the projection method developed by De Jong and Van Straten \cite{DJVS}.

\section{Reviewing Weierstrass points}\label{recap}

We recall that a numerical semigroup $\H$ is a subset of the nonnegative integers $\mathbb{N}$
containing $0$, closed under addition such that only a finite number of elements are missing from $\H$.  The genus of $\H$ is the number of its gaps, i.e. the number of
positive integers that are not in $\H$, $$g(\H):=\#(\mathbb{N}\setminus\H)=\#\{1=\l_1<\dots<\l_g\},$$ and we easily see that
the largest gap $\l_g$ is not bigger than $2g-1$. 

Given an irreducible smooth pointed curve $(\C,P)\in\mathcal{M}_{g,1}$ of genus $g$, its associated Weierstrass semigroup $\H_{P}$ is the subset of all 
non-negative integers $n$ such that 
\begin{equation}\label{nongap}
\mathrm{H}^{0}(\C,\LO_{\C}((n-1)P))\subsetneq\mathrm{H}^{0}(\C,\LO_{\C}(nP)),
\end{equation}
 i.e. $n\in\H_{P}$ if and only if there is a rational function on $\C$ whose pole divisor is exactly $nP$. The point $P\in\C$ is 
called a Weierstrass point if $\H_{P}$ is different from the ordinary semigroup $\{0, g+1,g+2,\dots\}$. The 
Riemann--Roch Theorem implies that the genus of the Weierstrass semigroup $\H_P$ is equal to the genus of the curve $\C$.  It is well known that only a finite number of Weierstrass
points exist on a curve. 

Since the $i$-th gap of $\H$ defines an upper semicontinuous function on $\mathcal{M}_{g,1}$, it follows that $\MS$ is a locally closed subset of $\mathcal{M}_{g,1}$. However, it is also well known that the moduli space $\MS$ can be empty, meaning that there are numerical semigroups that cannot be realized as Weierstrass semigroups of a smooth pointed curve. There is no purely arithmetical criterion for determining when a numerical semigroup is realizable, but one necessary numerical condition is given by Buchweitz in \cite{Bu76}.

On the other hand, one can see that any numerical semigroup can be realized as a Weierstrass semigroup of a monomial curve. Taking $\H := <a_1, \dots, a_r>$, a numerical semigroup, let $\k[\H] := \oplus_{n \in \H} \k t^n$ be the associated semigroup algebra. The affine monomial curve attached to $\H$ is
\begin{equation}\label{monomial}
\C_{\H} = \Spec \k[\H] \subset \mathbb{A}^r,
\end{equation}
which in parametric terms is just
\[\C_{\H} = \{(t^{a_1}, \dots, t^{a_r}) \in \mathbb{A}^r \mid t \in \mathbb{A}^1\}.\]
It is easy to produce the closure of $\C_{\H}$ in a weighted projective space $\mathbb{P}^r$ by adding just a smooth point $P$ at infinity, and so the Weierstrass semigroup at $P$ is $\H$. Here, a Weierstrass point on an integral curve $\C$ at a smooth point $P$ is defined in the same way that when $\C$ is smooth, i.e., for a smooth point $P$ on $\C$, a positive integer $n$ is a nongap if and only if equation \eqref{nongap} holds.

A criterion for determine if a numerical semigroup $\H$ is realizable was given by Pinkham in his PhD Thesis. Namely, a numerical semigroup $\H$ is realizable if and only if the affine monomial
curve $\C_{\H}$ admits a negative smoothing, c.f. \cite[pg. 108]{Pi74}. Dealing with this criterion is unfortunately far from easy.

\section{Gorenstein curves and subcanonical points}

Throughout this section $\C$ stands for a non-hyperelliptic Gorenstein curve with a smooth subcanonical point $P$, i.e. the associated Weierstrass semigroup $\H$ at $P$ is symmetric.
Recall that a numerical semigroup is symmetric if
the Frobenius number $\l_{g}$ of $\H$ is the biggest possible, $\l_{g}=2g-1$. Equivalently,  $$\l_{g-i}=2g-1-n_{i}\ \ (0\leq i\leq g-1),$$
where $0=n_0<n_1<n_2<\dots$ are the nongaps of the semigroup. Since it is assumed $\H$ to be non-hyperelliptic, we may impose that $\l_2=2$, equivalently,
$n_{g-1}=2g-2$. 

We also fix at once a system of generators, $\H := <a_1, \dots, a_r>$. We are interested in two suitable systems of generators: the minimal system, where $r$ is the embedding dimension of $\H$, and the canonical system of generators, i.e., $r = g - 1$. In the former case, $\H$ is generated by its first $g$ nongaps, $\H = <n_0, n_1, \dots, n_{g - 1}>$.

As a general and important comment, if a curve $\C$ is a local complete intersection, then $\C$ is also Gorenstein and non-hyperelliptic, because the dualizing sheaf of a local complete intersection always induces an embedding.

\subsection{On $P$-hermitian bases}\label{phermitian}

By virtue of the Max Noether Theorem for non-hyperelliptic Gorenstein curves (\cite{CS}), the maps
\begin{equation}\label{maxnoether}
\mathrm{Sym}^n\,\mathrm{H}^0(\C,\omega)\longrightarrow\mathrm{H}^{0}(\C,\omega^n)
\end{equation}
are surjective for all $n\geq 1$, where $\omega\cong \mathcal{O}_{\C}((2g-2)P)$ is the dualizing sheaf of $\C$.  Hence, each vector space $\mathrm{H}^0(\C,\omega^n)$ admits a so-called $P$-hermitian basis, i.e.
given $n\geq 1$, for each nongap $\H\leq n(2g-2)$ we can choose a meromorphic function on $\C$ of the form
$\bfx^{\balpha}_{s}:=x_{a_0}^{\alpha_0}\dots x_{a_r}^{\alpha_r}$ satisfying  $$\mathrm{ord}_{\infty, P}\bfx^{\balpha}_{s}=\sum\alpha_{i}a_{i}=s,$$
where each $x_{a_i}$ is a regular function on $\C\setminus\{P\}$ whose pole
order at $P$ is $\mathrm{ord}_{\infty, P}(x_{a_i})=a_i$. We also may declare
$a_0=0$, so one can assume that $x_{a_0}=1$. Hence, each
$\mathrm{H}^{0}(\C,\omega^n)$ admits a base formed by
meromorphic functions on $\C$ whose pole orders at $P$ are pairwise distinct.

In order to have a uniqueness between the chosen basis
elements $\mathbf{x}^{\balpha}_{s}$, one can take them
in a way that $\balpha:=(\alpha_0,\dots,\alpha_r)\in\mathbb{N}^{r+1}$ is a minimal
element  according to the lexicographical order 
\begin{equation}\label{lexi}
\left(\sum_{i=0}^{r} \alpha_i,\sum_{i=0}^{r} a_i\alpha_{i},-\alpha_{0},-\alpha_{r-1},\dots,-\alpha_{1}\right).
\end{equation}
Hence 
\begin{equation}\label{basewn}
\mathrm{H}^{0}(\C,\,\omega^n)=\mathrm{Spam}\bigcup_{s\leq n(2g-2)}\{\mathbf{x}_{s}^{\balpha}\,;\,\balpha \mbox{ is minimal}\}.
\end{equation}

For each $n\geq 1$, $\Delta_{n}$ stands for the vector subspace of $\k[X_{a_0},\dots,X_{a_r}]$ spanned by the lifting 
of the above monomial basis of $\mathrm{H}^{0}(\C,\,\omega^n)$, namely
\begin{equation}\label{basisn}
\Delta_{n}:=\mbox{Spam}\left(\bigcup_{s}\left\lbrace \bm{X}^{\balpha}_{s}\,;\,s\leq n(2g-2)\mbox{ and } \balpha \mbox{ minimal }\right\rbrace\right),
\end{equation} where $\mathbf{X}^{\balpha}_{s}:=X^{\alpha_0}_{a_0}\dots X_{a_r}^{\alpha_r}$, with $\sum_{i=1}^{r}a_i\alpha_i=s$.

We define $\deg(X_{a_i})=a_i$. It follows from the Riemann--Roch Theorem for singular curves that $\dim_{\k}\Delta_{n}=(2n-1)(g-1)$ and so
\begin{equation}\label{dimDeltan}
\dim\,\k[X_{a_0},\dots,X_{a_r}]_{\leq n}=\left(\begin{array}{c}
n+g-1\\
n
\end{array}\right)-(2n-1)(g-1),
\end{equation}
where
$\k[X_{a_0},\dots,X_{a_r}]_{\leq n}$ stands for vector space over $\k$ given by the isobaric polynomials of (weighted) degree not bigger than $n(2g-2)$. 

\begin{rmk}\label{rem1}
Considering the canonical system of generators for $\H = <n_0, n_1, \dots, n_{g - 1}>$, the above process produces a basis for $\Delta_n$ (respectively for $\mathrm{H}^0(\C, \omega^n)$) that is formed just by monomials on $X_{n_i}$ (respectively on $x_{n_i}$) all of the same degree $n$, which does not happen when we consider the minimal system of generators. For instance, the base elements of $\Delta_2$, respectively $\Delta_3$, are given by quadratic forms $X_{a_s} X_{b_s}$ with $\H \leq 4g - 4$, respectively by cubic forms $X_{u_\sigma} X_{v_\sigma} X_{w_\sigma}$ with $\sigma \leq 6g - 6$, according to the order fixed in \eqref{lexi}. We may also conclude that
$$\dim \,\k[X_{n_0}, \dots, X_{n_{g - 1}}]_n = \binom{n + g - 1}{n} - (2n - 1)(g - 1),$$
where $\k[X_{n_0}, \dots, X_{n_{g - 1}}]_n$ stands for the usual vector space given by the forms of degree $n$.
\end{rmk}

\subsection{The ideal of the canonically embedded $\C$}\label{idealcan}

We start by identifying $\C$ with its image under
the canonical embedding given by its dualizing sheaf $\omega$. In this way, $\C$ can be viewed
as a curve of genus $g$ and degree $2g-2$ in $\mathbb{P}^{g-1}$. Let $I(\mathcal{C})=\oplus_{j=2}^{\infty}I_j(\mathcal{C})$ be the homogeneous ideal of $\mathcal{C}.$ By Riemann's Theorem, for each $j\geq 2,$ the codimension of $I_j(\mathcal{C})$ in the vector space $\k[X_{n_0},\dots,X_{n_{g-1}}]_{j}$ is $(2j-1)(g-1)=dim_{\k}\Delta_j$. Then we obtain $$\k[X_{n_0},\dots,X_{n_{g-1}}]_{j}=\Delta_j\oplus I_j(\mathcal{C}), \quad \textrm{ for each } j\geq 2.$$
Recall, see \cite[Theorems 1.7 and 1.9]{O91}, that each nongap $\H\leq 4g-4$ can be written in $\nu_s$ different ways
as a sum of two non gaps not bigger than $2g-2$, namely
$$\H=a_{s1}+b_{s1}=\dots=a_{s\nu_s}+b_{s\nu_s},$$ where $a_{s_1} < \dots < a_{s_{\nu_s}}$ and $a_{s_i} \leq b_{s_i}, $ for all $i=1,\dots, \nu_s$.
Analogously,
there are $\nu_{\sigma}$ different ways to write each nongap $\sigma\leq 6g-6$ as a sum of three nongaps,
$$\sigma=u_{\sigma1}+v_{\sigma1}+w_{\sigma1}=\dots=u_{\sigma\nu_{\sigma}}+v_{\sigma\nu_{\sigma}}+w_{\sigma\nu_{\sigma}}.$$
Since $x_{a_{si}}x_{b_{si}}\in H^0(\C,\omega^2)$ and $x_{u_{\sigma j}}x_{v_{\sigma j}}x_{w_{\sigma j}}\in H^0(\C,\omega^3)$, we
may assume that $x_{a_{s1}}x_{b_{s1}}:=x_{a_s}x_{b_s}$ and  $x_{u_{\sigma 1}}x_{v_{\sigma 1}}x_{w_{\sigma 1}}:=x_{u_{\sigma}}x_{v_{\sigma}}x_{w_{\sigma}}$ are base elements of $\Delta_2$ and $\Delta_3$, respectively, c.f. Remark \ref{rem1}. Hence, for each $i=2,\dots,\nu_{s}$ and 
each $j=2,\dots,\nu_{\sigma}$ the elements $x_{a_{si}}x_{b_{si}}$ and
$x_{u_{\sigma j}}x_{v_{\sigma j}}x_{w_{\sigma j}}$ can be written
as a linear combination of base elements, preserving the pole order at $P$, namely
\begin{eqnarray}
x_{a_{si}}x_{b_{si}}=c_{sis}x_{a_{s}}x_{b_{s}}+\sum_{n < s}c_{sin}x_{a_{n}}x_{b_n} \\
x_{u_{\sigma j}}x_{v_{\sigma j}}x_{w_{\sigma j}}=d_{\sigma j\sigma}x_{u_{\sigma}}x_{v_{\sigma}}x_{w_{\sigma}}+\sum_{m<\sigma}d_{\sigma jm}x_{u_{m}}x_{v_{m}}x_{w_{m}}
\end{eqnarray}
where $n$ and $m$ run over the nongaps and $c_{sin}$, $d_{\sigma jm}\in\k$ are constants. We also may assume that $c_{sis}=d_{\sigma j\sigma}=1$, because they must be different from
zero and so we can multiply them by suitable constants.
By construction, the $\frac{1}{2}(g-2)(g-3)$ quadratic forms 
\begin{eqnarray}\label{formsquad}
F_{si}:=X_{a_{si}}X_{b_{si}}-X_{a_s}X_{b_s}-\sum_{n=0}^{s-1} c_{sin}X_{a_n}X_{b_n}\in\k[X_{n_0},\dots,X_{n_{g-1}}]
\end{eqnarray}
and the $\left(\begin{array}{c}
     g+2\\
     3
\end{array}\right)-(5g-5)$ cubic forms
\begin{eqnarray}\label{cubicforms}
G_{\sigma j}=X_{a_{\sigma j}}X_{b_{\sigma j}}X_{c_{\sigma j}}-X_{a_{\sigma}}X_{b_{\sigma }}X_{c_{\sigma }}-\sum_{n=0}^{\sigma-1} d_{\sigma jn}X_{a_{\sigma}}X_{b_{\sigma }}X_{c_{\sigma }}\in\k[X_{n_0},\dots,X_{n_{g-1}}],
\end{eqnarray}
vanish identically on the canonical curve $\C$, are linearly independent and because of their number,  form a basis of the vector spaces $I_2(\C)$ and $I_3(\C)$, respectively. 

Petri's Analysis remains true for canonical Gorenstein curves and assures that the ideal $I(\C)$ is generated by quadratic relations, provided $\C$ is non-hyperelliptic, nontrigonal and not isomorphic to a quintic plane curve. When $\C$ is trigonal or isomorphic to a quintic plane curve, it assures that $I(\C)$ is generated by quadratic and cubic forms. It turns out that if $\H$ is such that $3<n_1<g$ and $\H\neq <4,5>$, then $\C$ is non-trigonal and not isomorphic to a quintic plane curve. Hence the ideal $I(\C)$ is generated by the quadratic forms in equation \eqref{formsquad}, see \cite{St93} or Theorem \ref{TeoDoKarl} below. Therefore, if $\H$ is such that $n_1=3$, $n_1=g$ or $\H=<4,5>$, then the ideal $I(\C)$ is generated by the quadratic forms in \eqref{formsquad} and suitable cubic forms picked up from  \eqref{cubicforms}, c.f. \cite[Thm 3.7] {CF}.

\medskip

It is worth to note that each non-hyperelliptic numerical semigroup $\H$ can be rea\-li\-zed as the Weierstrass semigroup of a Gorenstein (canonical) curve. Namely, taking the canonical monomial curve
\begin{equation}\label{defc0}
\C^{(0)}:=\{(a^{n_0}b^{\ell_g-1}: a^{n_1}b^{\ell_{g-1}-1}: \cdots:a^{n_{g-1}}b^{\ell_1-1} \,|\, (a:b)\in\mathbb{P}^1)\}\subset \mathbb{P}^{g-1},
\end{equation}
\noindent the Weierstrass semigroup at $P=(0:\cdots:0:1)$ is equal to $\H$, c.f. \cite[p.190]{St93}. Moreover, the ideal of $\C^{(0)}$ is generated by the following $\frac{1}{2}(g-2)(g-3)$ \emph{folded quadratic forms} (see \cite[Lemma 2.3] {CS})
\begin{equation}\label{geradoras}
F^{(0)}_{si}=X_{a_{si}}X_{b_{si}}-X_{a_s}X_{b_s},
\end{equation}
provided that $3<n_1<g$ and $\H\neq <4,5>$. In addition, if $n_1=3$, $n_1=g$ or $\H=<4,5>$, then the ideal of $I(\C^{(0)})$ is generated by the above $\frac{1}{2}(g-2)(g-3)$ folded quadratic
forms and suitable \emph{folded cubic forms}
\begin{eqnarray}\label{geradoras3}
G_{\sigma j}^{(0)}=X_{a_{\sigma j}}X_{b_{\sigma j}}X_{c_{\sigma j}}-X_{a_{\sigma}}X_{b_{b_{\sigma }}}X_{c_{\sigma }},
\end{eqnarray}  c.f. \cite[Lemma 3.3]{CF}.

\subsection{Unfolding the defining equations}\label{secunfod}

Given the monomial curve $\C_{\H}\subset\mathbb{A}^r$ associated
to any non-ordinary semigroup $\H=<a_1,\dots,a_r>$, a result due to Herzog, \cite{Her70}, assures that the generators of the ideal $I(\C_{\H})$ can be chosen to be \textit{isobaric forms} which are given by the difference of two monomials in the variables $X_{a_1},\dots,X_{a_r}$, namely
\begin{equation}\label{eqherz}
  \mathbf{X}_{s}^{\boldsymbol{\alpha}}-\mathbf{X}_{s}^{\boldsymbol{\beta}},
\end{equation}
such that $\alpha_i\beta_i=0$ for $i=1,\dots,r$ and $\sum a_i\alpha_i=\sum a_i\beta_i=s$ is its weight.

When we assume non-hyperelliptic symmetric semigroups, we can consider two systems of generators for $\H$, namely the minimal and the canonical ones. By considering the minimal system of generators, $a_1,\dots,a_r$, we may take the basis of $\Delta_i\subset \k[X_{a_0},X_{a_1},\dots,X_{a_r}]$, for $i\geq 2$, which is given by the lifting of the P-hermitian
basis of  $\mathrm{H}^0(C,\omega^i)$, where $\omega$ is the dualizing sheaf of $\C_{\H}$, see Eq. \eqref{basisn} of Section \ref{phermitian}. If $H^{(0)}$ is 
a generating form of $I(\C_{\H})$, say
$H^{(0)}=\mathbf{X}_{s}^{\boldsymbol{\alpha}}-\mathbf{X}_{s}^{\boldsymbol{\beta}}\in I(\C_{\H})$ of weight $\H$, let $n$ be the smallest positive integer such that $\H\leq n(2g-2)$. Thus the \textit{unfold} of $H^{(0)}$ is the polynomial
\begin{equation}\label{defunfold}
   H_s=H^{(0)}_{s}+\sum_{j<s}c_{sj}\mathbf{X}^{\boldsymbol{\gamma}}_{j}\vert_{X_{a_0}=1}\in\k[\{c_{sj}\}]\otimes\k[X_{a_1},\dots X_{a_r}],
\end{equation}
where each $\mathbf{X}^{\boldsymbol{\gamma}}_{j}$ is the unique basis element of $\Delta_{n}$ of weight $j$,  
$\mathbf{X}^{\boldsymbol{\gamma}}_{j}\vert_{X_{a_0}=1}$ is the monomial obtained from $\mathbf{X}^{\boldsymbol{\gamma}}_{j}$ making $X_{a_0}=1$ and $c_{sj}$ are variables over the ground field $\k$. We attach weight $\H-j$ to each $c_{sj}$. Since the weight of $X_{a_0}$ is zero, the unfold of $H^{(0)}$ is also an isobaric form of degree $\H$.

By considering the canonical system of generators for $\H$, namely $n_0,n_1,\dots,n_{g-1}$, the canonical ideal of $\C^{(0)}$ is generated by isobaric forms that are also homogeneous polynomials (quadratic and cubic) in the usual sense, c.f. Eq. \eqref{geradoras} and Eq. \eqref{geradoras3}. Then the \textit{unfold} a defining quadratic form $F_{si}^{(0)}$ is
\begin{eqnarray}\label{unfoldc0}
F_{si}=F_{si}^{(0)}-\sum_{n=0}^{s-1} c_{sin}X_{a_n}X_{b_n}\in\k[\{c_{sij}\}]\otimes\k[X_{n_0},\dots,X_{n_{g-1}}],
\end{eqnarray} while the \textit{unfold} of a cubic defining form $G_{si}^{(0)}$ is
\begin{eqnarray}\label{unfodcubic}
G_{\sigma j}=G_{\sigma j}^{(0)}-\sum_{n=0}^{\sigma-1} d_{\sigma jn}X_{a_{\sigma}}X_{b_{\sigma }}X_{c_{\sigma }}\in\k[\{d_{\sigma ij}\}]\otimes\k[X_{n_0},\dots,X_{n_{g-1}}].
\end{eqnarray}
Note that the unfold of the quadratic and cubic defining equations of the canonical curve $\C^{(0)}$, are again quadratic and
cubic forms in  $\k[X_{n_0},\dots,X_{n_{g-1}}]$ and isobaric as well, provided
the weight of $c_{sin}$ and $d_{\sigma jn}$ are $\H-n$ and $\sigma-n$, respectively.

It is evident that the unfold of the defining equations of a monomial curve is a perturbation of its defining ideal. To get a deformation preserving at least the dimension and the arithmetical genus over the fibers, these perturbations cannot be chosen independently. Generally, they are related by \emph{syzygetic} relations. This is precisely the subject of the next subsection.

\subsection{A variant of Hauser's algorithm}\label{secHauser}

In his PhD thesis \cite{Pi74}, Pinkham constructs the moduli space $\MS$ using equivariant (versal) deformation theory. In short, Pinkham starts by considering the versal deformation space of the affine monomial curve $\C_{\H}$, say 
$$
\begin{matrix} 
\mathcal{X}_{t_0}\cong\C_{\H} & \longrightarrow & \mathcal{X} \\[3pt] 
\Big\downarrow & & \Big\downarrow\\[7pt]
\{t_0\}=\Spec\,\k & \longrightarrow & \mathcal{T}
\end{matrix}
$$
where $\mathcal{T}=\Spec A$ and $A$ is a local, complete noetherian $\k$-algebra, c.f. \cite{Ar76}.
The $\mathbb{G}_{m}$-action on $\C_{\H}$,
given by  $(\zeta, X_{a_i})\mapsto \zeta^{a_i}X_{a_i}$,  
can be extended to the total  and
parameter spaces, $\mathcal{X}$ and  $\mathcal{T}$, inducing a
grading on the  tangent  space $\mathbf{T}^1_{\C_{\H}}\cong \mathbf{T}^{1}(\k[\H])$ to
$\mathcal{T}$, that is the cotangent complex associated to $\C_{\H}$. 

We declare that a deformation has negative weight $-e$ if it decreases the weights of the defining equations of the curve and the corresponding deformation variable
has then (positive) weight $e$. It is more than convenient to note
that the unfolds of the defining forms of $\C_{\H}$ and $\C^{(0)}$ in Equations \eqref{defunfold}, \eqref{unfoldc0} and \eqref{unfodcubic} of the preceding subsection occur in negative degrees, once provided that they define a deformation of $\C_{\H}$ and $\C^{(0)}$, respectively.


Let $\mathbf{I}$ be the ideal of $A$ generated by the elements corresponding to the positive graded part $\mathbf{T}^{1,+}(\k[\H])$. 
The space $\mathcal{T}^-:=\Spec {A}/\mathbf{I}$  is the subspace of $\mathcal{T}$ in 
negative degrees and the restriction $\mathcal{X}^- \to \mathcal{T}^-$ is the 
versal deformation in negative degrees,
$$
\begin{matrix} 
\mathcal{X}_{t_0}\cong\C_{\H} & \longrightarrow & \mathcal{X}^{-} \\[3pt] 
\Big\downarrow & & \Big\downarrow\\[7pt]
\{t_0\}=\Spec\,\k & \longrightarrow & \mathcal{T}^{-}=\Spec(A/\mathbf{I})
\end{matrix}
$$

In addition, the total space $\mathcal{X}^- $ and the parameter space $ \mathcal{T}^-$  are both defined by polynomials. In general, the total and parameter spaces associated to an analytic singularity cannot be defined by polynomial equations alone, and sometimes do not have a finite dimension. However, this does not happen when deforming quasi-homogeneous singularities.

Next Pinkham produces a fiberwise compactification $\smash{\overline{\mathcal{X}}}^- \to \mathcal{T}^-$ of the versal deformation in negative degrees $\mathcal{X}^{-}\rightarrow \mathcal{T}^{-}$ without compactify the parameter and total space, avoiding technical problems coming from inverse limits. Doing this, Pinkham shows that each fiber of $\smash{\overline{\mathcal{X}}}^- \to \mathcal{T}^-$ is an integral curve
in a weighted projective space with one point $P$ at infinity whose associated Weierstrass semigroup is exactly $\H$.
All the fibers over a given
$\mathbb{G}_m$ orbit of $\mathcal{T}^-$ are isomorphic, and two fibers
are isomorphic if and only if they lie in the same orbit.
This is proved in \cite{Pi74} for smooth fibers and in general in the Appendix
of \cite{Lo84}.

Now let us invert the above considerations starting with a possible singular integral curve $\C$, of arithmetic genus $g>1$ 
defined over 
$\k$. 
Given a smooth point $P$ of $\C$, let  $\H$ be the Weierstrass semigroup of $\C$ at $P$. Consider the line  bundle $L=\LO_{\C}(P)$ and
form the ring of sections $\mathcal{R} = \oplus_{i=0}^\infty H^0(\C,L^i)$.
This leads to an embedding of  $\C= \mathbb{P} (\mathcal R)$ in a weighted projective space, with coordinates $X_{a_0},  \dots, X_{a_r}$ with $\deg (X_{a_0}) = 1$.  The space $\Spec  \mathcal R$
is the corresponding quasi-cone in affine space. Setting $X_{a_0} = 0$ defines the monomial 
curve $\C_{\H}$, all other fibers are isomorphic to $\C\setminus P$. In particular, if $\C$ is 
smooth, this construction defines a smoothing of $\C_{\H}$.
Then Pinkham establishes the following result:

\begin{thm}[\null{\cite[Thm. 13.9]{Pi74}}]\label{pinkhamthm}
Let $\mathcal{X}^- \to \mathcal{T}^-$ be the equivariant 
miniversal deformation in negative
degrees of the monomial curve $\C_{\H}$ for a given semigroup $\H$ and denote
by $\mathcal{U}^-$ the open subset of $\mathcal{T}^-$ given by the points
with smooth fibers. Then the moduli space $\MS$ is isomorphic to the quotient
$\MS=(\mathcal{U}^-)/\mathbb{G}_{m}$ of  $\mathcal{U}^-$ 
by the  $\mathbb{G}_{m}$-action.
\end{thm}

The remainder of this subsection is devoted to explicitly describing Pinkham's Theorem \ref{pinkhamthm} in the case where $\H$ is assumed to be a non-hyperelliptic symmetric semigroup. We present the construction initiated by Stoehr \cite{St93}, and subsequently developed by Contiero-Stoehr \cite{CS} and Contiero-Fontes \cite{CF}. This construction can be viewed as a variant of Hauser's algorithm for computing the versal deformation space of a singularity; see \cite{Hau83,Hau85} and \cite{Stev13}.

We start by fixing a numerical symmetric semigroup $\H=<n_0,n_1,\dots,n_{g-1}>$ of genus $g>3$ satisfying $$3<n_1<g \ \ \text{and} \ \ \H\neq<4,5>.$$ These restrictions are also imposed to avoid simple and well-known cases. If $n_1=3$ or $\H=<3,4>$, then $\C_{\H}$ is a plane curve. Additionally, if $n_1=g$, then $\H$ is not a complete intersection, see \cite{BMG}, and the
associated moduli it is studied in \cite{CF}.

So, if $\C$ is a Gorenstein curve with a smooth point whose associated Weierstrass semigroup is $\H$, then $\C$ can be identified with this image under the canonical
embedding in such a way that the Weierstrass point $P$ that
realizes $\H$ is the point $P=(0:\ldots:0:1)$. Hence by section \ref{idealcan}, the canonical ideal $I(\C)\subset\k[X_{n_0},\dots,X_{n_{g-1}}]$ is generated by the
$\frac{1}{2}(g-2)(g-3)$ quadratic forms
$$
F_{si}=X_{a_{si}}X_{b_{si}}-X_{a_s}X_{b_s}-\sum_{n=0}^{s-1} c_{sin}X_{a_n}X_{b_n}\in\k[X_{n_0},\dots,X_{n_{g-1}}],$$
where $c_{sij}$ are suitable constants in $\k$, the forms
$F_{si}^{(0)}=X_{a_{si}}X_{b_{si}}-X_{a_s}X_{b_s}$ generate the ideal of the canonical monomial curve $\C^{(0)}\subset\mathbb{P}^{g-1}$ defined in \eqref{defc0}, and each $X_{a_j}X_{b_j}$ belongs to the fixed base $\Delta_2$ in equation \eqref{basisn}.

Now let us invert the considerations on the previous paragraph. Let 
$$F_{si}^{(0)}=X_{a_{si}}X_{b_{si}}-X_{a_s}X_{b_s}$$ be the defining polynomials of the canonical curve $\C^{(0)}$ as in \eqref{geradoras}. Now let us take their unfolding
$$
F_{si}=F_{si}^{(0)}-\sum_{n=0}^{s-1} c_{sin}X_{a_n}X_{b_n}\in\k[\{c_{sij}\}]\otimes\k[X_{n_0},\dots,X_{n_{g-1}}],
$$defined in section \ref{secunfod} equation \eqref{unfoldc0}. We want to determine the constants $c_{sin}$ in order that the intersection of the $V(F_{si})$ in $\mathbb{P}^{g-1}$ is a canonical Gorenstein curve of genus $g$ whose Weierstrass semigroup at the smooth point $P$ is $\H$. 

Since the coordinates functions $x_n$ introduced in section \ref{phermitian}, where $n\in\H$ 
and $n\leq 2g-2$, are not uniquely determined by their pole divisor $nP$,
we may transform
\[
X_{n_i}\longmapsto X_{n_i}+\sum_{j=0}^{i-1}\alpha_{ij}X_{n_{i-j}},
\]
for each $i=1,\ldots, g-1$, and so we can normalize 
$\frac{1}{2}g(g-1)$ of the coefficients $c_{sin}$ to be zero, see 
\cite[Proposition 3.1]{St93}. Due to these normalizations and the 
normalizations of the coefficients $c_{sin}=1$ with $n=s$, the only freedom left to us 
is to transform $x_{n_i}\mapsto \alpha^{n_i}x_{n_i}$ for $i=1,\ldots, g-1$.

The first step to the explicit construction of a compactification of $\MS$ due to Contiero--Stoehr is the following lemma:

\begin{lem}[Syzygy Lemma \cite{CS}] \label{syzygylem}
For each of the $\frac{1}{2}(g-2)(g-5)$ quadratic binomials $F_{s'i'}^{(0)}$ different from $F_{n_i+2g-2,1}^{(0)},\, i=0,\dots,g-3$, there is a syzygy of the form  
\begin{equation}\label{sygies}
X_{2g-2}F^{(0)}_{s'i'}+\sum_{nsi}\epsilon_{nsi}^{(s'i')}X_nF^{(0)}_{si}=0
\end{equation}where the coefficients $\epsilon^{s'i'}_{nsi}$ are integers equal to $1,-1$ or $0$ and where the sum is taken over the nongaps $n\leq 2g-2$ and the double indices $si$ with  $n+s=2g-2+s.$
\end{lem}

The algorithmic construction of the closure of $\MS$ starts by replacing 
the initial binomials $F_{s'i'}^{(0)}$ and $F_{si}^{(0)}$ in equation \eqref{sygies} 
by the corresponding 
 unfolded 
forms $F_{s'i'}$  and
$F_{si}$ displayed in equation \eqref{unfoldc0} of section \ref{secunfod}, obtaining a 
linear combination of cubic monomials of weight smaller than $\H'+2g-2$. By virtue of \cite[Lemma 2.4]{CS}  and its proof,
this linear combination of cubic monomials admits the following decomposition:
\begin{equation*}
 X_{2g-2}F_{s'i'}+\sum_{nsi}\varepsilon_{nsi}^{(s'i')}X_{n}F_{si}=
 \sum_{nsi}\eta_{nsi}^{(s'i')}X_{n}F_{si}+R_{s'i'},
\end{equation*}
where the sum on the right hand side is taken over the nongaps $n\leq 2g-2$ and
the double indexes $si$ with $n+s<s'+2g-2$, the coefficients
$\eta_{nsi}^{(s'i')}$ are constants, and where $R_{s'i'}$ is a linear
combination of cubic monomials of pairwise different weights smaller than $\H'+2g-2$.

For each nongap $m<s'+2g-2$, let $\varrho_{s'i'm}$ be the unique coefficient
of $R_{s'i'}$ of weight $m$. It is a quasi-homogeneous polynomial expression
of weight $\H'+ 2g - 2 - m$ in the coefficients $c_{sin}$.

All the objects that are required to construct the compactification of $\MS$ were introduced above. The main results due to Stoher and Contiero-Stoehr are the following.

\begin{thm}[c.f. Thm 2.6 of \cite{CS}]\label{TeoDoKarl}
Let $\H\subset\mathbb{N}$ be a numerical symmetric semigroup of genus $g$ satisfying $3<n_1<g$ and $\H\neq <4,5>$. Then the $\frac{1}{2}(g-2)(g-3)$ quadratic forms $F_{si}=F_{si}^{(0)}-\sum_{n=0}^{s-1}c_{sin}X_{a_n}X_{b_n}$ cut out a canonical integral Gorenstein curve in $\mathbb{P}^{g-1}$ if and only if the coefficients $c_{sin}$ satisfy the quasi-homogeneous equations $\varrho_{s'i'm}=0.$ In this case, the point $P=(0:0:\cdots:1)$ is a smooth point of the canonical curve with Weierstrass semigroup $\H$.

\end{thm}

\begin{thm}[c.f. Thm. 2.7 of \cite{CS}]
Let $\H\subset\mathbb{N}$ a symmetric semigroup of genus $g:=\#(\mathbb{N}\setminus S)$ satisfying $3<n_1<g$ and $\H\neq <4,5>$. The isomorphism classes of the pointed complete integral Gorenstein curves with Weierstrass semigroup $\H$ correspond bijectively to the orbits of the $\mathbb{G}_m(\k)$-action
$$(c,\cdots,c_{sin},\cdots)\mapsto (\cdots,c^{s-n}c_{sin},\cdots)$$ on the affine quasi-cone of the vectors whose coordinates are the coefficients $c_{sin}$ of the normalized quadratic forms $F_{si}$ that satisfy the quasi-homogeneous equations $\sigma_{s'i'm}=0.$
\end{thm}

\section{The Main Theorem}

Let $\H$ be a non-hyperelliptic symmetric semigroup and $0=n_0< n_1<\dots< n_{g-1}$ its canonical system of generators. Let us also take $a_1<\dots<a_r$ a minimal system of generators of $\H$. Considering the polynomial rings $\k[X_{n_0},\cdots,X_{n_{g-1}}]$ and $\k[X_{a_1},\dots,X_{a_r}]$, we attach to the variable $X_k$ the degree $\deg(X_k)=k$ and then $\deg(X_k^{\alpha})=\alpha\cdot\deg(X_k)$. The map  

$$\begin{array}{cccc}
\Pi: & \k[X_{n_0},\dots,X_{n_{g-1}}]  & \longrightarrow & \k[X_{a_1},\dots,X_{a_r}] \\
    & X_{n_0} & \longmapsto & 1\\
    & X_{n_i} & \longmapsto & \mathbf{X}^{\boldsymbol{\alpha}}_{n_i}
\end{array},$$ where $\mathbf{X}^{\boldsymbol{\alpha}}_{n_i}$ is the monomial in the variables $X_{a_i}$ introduced in above section \ref{phermitian},
is a graded homomorphism between $\k[X_{n_0},\cdots,X_{n_{g-1}}]$ and $\k[X_{a_1},\dots,X_{a_r}]$. Henceforward, the \textit{shrinking map} stands to this homomorphism $\Pi$.

\subsection{Proof of the Main Theorem}

Let $\H$ be a numerical symmetric non-hyperelliptic semigroup of genus $g>1$. Then $\H$ is realized as the Weierstrass semigroup of the canonical monomial curve $\C^{(0)}$ at the point $P=(0:\cdots:0:1)$. Considering the affine open chart $X_0=1,$ the parametrization of $\C^{(0)}|_{X_0=1}$ is given by $$\C^{(0)}|_{X_0=1}=\left\{(t^{n_1},t^{n_2},\cdots,t^{n_{g-1}})| \, t\in\mathbb{A}^1\right\}\subset\mathbb{A}^{g-1}.$$ 
On the other hand, let $a_1,\cdots,a_r$ be the minimal system of generators of $\H$ and consider the affine monomial curve $\mathcal{C}_S=\mathrm{Spec}\,\k[S]$.
Then $\C_S\simeq \C^{(0)}|_{X_0=1}=\C^{(0)}\setminus P$ because their coordinate rings are both $\k[\H]$. As the projetivization of $\C^{(0)}\setminus P$ and $\C_S$ consists by adding a single point $P=(0:\cdots:0:1)\in\mathbb{P}^{g-1}$ and $Q=(0:\cdots:0:1)\in\mathbb{P}^r$, respectively, at infinity, we conclude that the projetivization of this two curves are also isomorphic. Since we do not lose information on the coefficients of the unfolded quadratic forms that generates the monomial curve $\C^{(0)}$, we can adapt Stoehr's construction to provided a compactification of $\MS$ using the monomial affine curve $\C_S$ instead of the canonical one $\C^{(0)}$. To do this we shrink all the forms that are
involved in Stoehr's construction, in particular the $P-$hermitian basis and the unfolded quadratic forms.

Let us fix an algebraic closed field $\k$  of arbitrary characteristic. In order to prove Theorem B, we first shall prove the following theorem:

\begin{thm}\label{Teo1}
Let $\C$ a non-hyperelliptic Gorenstein curve defined over $\k$ and $\H$ a complete intersection numerical semigroup. Then $\C$ realizes $\H$ at a smooth point $P\in\C$ if and only if there is an embedding of $\C$ into $\mathbb{P}^r$ such that the defining equations of $\C\setminus P$ are given by the unfolding of the $r-1$ defining equations of $\C_S\subseteq\mathbb{A}^{r}$.
\end{thm}

\begin{proof}Let $\C$ be a complete integral Gorenstein curve and $P$ be a smooth point on $\C$ whose Weierstrass semigroup is equal to $\H$. Let us take the line  bundle $\mathcal{L}=\mathcal{O}_{\C}(P)$ and its associated
ring of sections $\mathcal{R} = \oplus_{i=0}^\infty H^0(\C,\mathcal{L}^i)$.
Since we are fixing a minimal system of generators for $\H$, the ring $\mathcal{R}$ induces an embedding of $\C= \mathbb{P}(\mathcal R)$ in a weighted projective space, with coordinates $Y_{0},\dots, Y_{r}$ with $\deg Y_{0} = 0$. 
The space $\Spec  \mathcal{R}$
is the corresponding quasi-cone in affine space. Setting $Y_0 = 0$ defines the monomial curve $\C_\H$ and all other fibers are isomorphic to $\C\setminus P$. In particular, $\C\setminus P$ is obtained by a deformation of $\C_{\H}$, as predicted by Pinkham's construction. Since  every deformation of $\C_{\H}$ that realizes $\H$ at an added point at the infinity is obtained by unfolding the defining equations of $\C_{\H}$, c.f. Theorem \ref{TeoDoKarl}, we are done.

Conversely, let $\mathcal{D}$ be an affine curve in $\mathbb{A}^r$ that is given by the unfold of the regular sequence $G_{kj}^{(0)}\in \k[X_{a_1},
\dots,X_{a_r}]$ that generates the ideal of $\C_{\H}$, where each $G_{kj}^{(0)}$ is an isobaric polynomial of degree $k$. So
the ideal of $\mathcal{D}$ is given by ${G_{kj}}=G_{kj}^{(0)}+
\sum_i e_i\beta_i,$ with $\beta_i\in\Gamma_2$, where $\Gamma_2$ is the shrink $P$-hermitian basis of $H^0(\C^{(0)},\omega^2)$ fixed in \eqref{basewn}, and $e_i\in \k$.

Following Stoehr's construction, a curve is in $\overline{\MS}$ if and only if satisfies some quasi-homogeneous equations $\rho_{s'i'n}=0$ come from
suitable syzygies of the generators $F_{si}^{(0)}$ of the affine monomial curve $\C^{(0)}$, c.f. Theorem \ref{TeoDoKarl}. Now, the Syzygy Lemma \ref{sygies} assures the existence of $\frac{1}{2}(g-2)(g-5)$ syzygies of the form $$\H_{s'i'}:= X_{2g-2}F^{(0)}_{s'i'}+\sum_{nsi}\epsilon_{nsi}^{(s'i')}X_nF^{(0)}_{si}=0.$$
Taking the image of theses syzygies under the shrink map $\Pi$, we get
$$\Pi(S_{s'i'})= \sum_{j=1}^{r-1}M_{s'i'j}G_{kj}^{(0)}=0,$$ where $M_{s'i'j}\in \k[X_{a_1},\dots,X_{a_r}]$ are isobaric polynomials of weight $\H'-k$. Using Koszul complex we observe that the relations between the generators $G_{ki}^{(0)}, i=1,\dots,r-1$, must be trivial, because $\C_{\H}$ is a complete intersection. Thus when we exchange $G_{kj}^{(0)}$ by the unfold ${G_{kj}},$ the relations $\varrho_{s'i'm}=0$ between the coefficients given by \ref{TeoDoKarl} are trivially satisfied. Then the projetivized unfolds of the forms that generates $\mathcal{D}$ cut out an integral curve in $\mathbb{P}^r$ with Weierstrass semigroup $\H$ in $Q(0:\dots:0:1)$. \end{proof}

\begin{proof}[Proof of the Main Theorem] 
By virtue of the above Theorem \ref{Teo1}, an integral curve is in $\overline{\MS}$ if
and only if it is given by the unfolding of the regular sequence of the complete intersection
monomial curve $\C_{\H}$. Then the space $\overline{\MS}$ is just determined by the coefficients of the unfolded forms. The coordinate functions $x_n$, for $n\in \H$, were choosing as functions with pole divisors $nP$, so they are not uniquely determined. Hence we are able to do the following changes of variables
$$X_n\mapsto X_n + \sum_{m=0}^{n-1}d_{nm}X_m,$$ where the coefficients $d_{nm}$ are constant. As there are $r$ minimal generators in $\H$ we can normalize $\frac{1}{2}r(r+1)$ coefficients with weights determined in the unfolds of the generator polynomials of the complete intersection affine curve, provided the characteristic of the ground field $\k$ is zero or not a prime divisor of any exponent of the defining equations of $\C_{\H}$. After these normalizations, the only change we can make is to transform $x_{a_i}\mapsto c^{a_i}x_{a_i}, i=1,\dots,r-1$, for some $c\in\mathbb{G}_m(\k)=\k^*$. According to \cite[Appendix]{St93} the coefficients of the normalized unfolded polynomials form a basis for the negatively-graded part of the first cohomology module of the cotangent complex $\mathbf{T}^{1,-}(\mathbf{k}[\H])$. Hence we conclude that $\overline{\MS}=\mathbb{P}(\mathbf{T}^{1,-}(\mathbf{k}[\H]).$
\end{proof}

\subsection{Examples}
\begin{exam}\label{exam1}
We start with a simple example in codimension $2$. Given a positive integer $\tau$, consider the semigroup
$$\H=<4,3+4\tau,6+4\tau> = 4\mathbb{N}\sqcup (3+4\tau+4\mathbb{N}) \sqcup (6+4\tau+4\mathbb{N}) \sqcup (9+8\tau+4\mathbb{N}),$$
 of genus $g=3+4\tau$ and whose Frobenius number is $\ell_g=5+8\tau =2g-1$, so $\H$ is symmetric. Consider the affine monomial curve $$\C_{\H}:=\{(t^4,t^{3+4\tau},t^{6+4\tau}), t\in \k\}\subset
   \mathbb{A}^3 $$
and $\C^{(0)}\subset \mathbb{P}^{g-1}$
the canonical monomial curve where $P=(0:\cdots:0:1)$ realizes $\H$. Let $\{x_0,x_{n_1},\cdots,x_{n_{g-1}}\}$ be a basis for $\mathrm{H}^{0}(\C^{(0)},\mathcal{O}(P))$. For short we use 
$$x:=x_4, \quad y_3:=x_{3+4\tau}\quad \mbox{and} \quad y_6:=x_{6+4\tau}. $$
Thus a $P-$hermitian base of $H^{0}(\C^{(0)}, \omega)=H^{0}(\C^{(0)},\mathcal{O}(4+8\tau)P)$ is given by
$$\left\{\begin{array}{l}
x^0, x,\cdots, x^{2\tau+1}\\
x^0y_3, xy_3,\cdots,x^{\tau}y_3\\
x^0y_6, xy_6,\cdots,x^{\tau-1}y_6
\end{array} \right., \qquad \tau\geq 1$$

\noindent We can consider $y_9:=x_{9+8\tau}$ as the product $y_3y_6$. Hence the $P-$hermitian basis for the bicanonical divisor $H^0(\C^{(0)}, \mathcal{O}(8+16\tau))$ is given by the $3g-3$ elements
$$\left\{
\begin{array}{l}
x^0,\cdots,x^{2+4\tau}\\
x^0y_3,xy_3,\cdots, x^{1+3\tau}y_3\\
x^0y_6,xy_6,\cdots, x^{3\tau}y_6\\
x^0y_3y_6,xy_3y_6,\cdots,x^{2\tau-1}y_3y_6
\end{array}\right.$$

\noindent Lifting the $P-$hermitian basis elements, we attach the variables $x,y_3$ and $y_6$ to $X, Y_3$ and $Y_6$ of weights $4,3+4\tau$ and $6+\tau$, respectively. For short we use $$Z_{4i}:=X^i, \quad Z_{j+4\tau+4i}:=X^iY_j, \quad Z_{9+8\tau+4i}:=X^iY_3Y_6.$$
\noindent As the curve $\mathcal{D}$ is complete intersection, there are two polynomials in $\k[X,Y_3,Y_6]$ that vanishing in $\mathcal{D}$ and generates its ideal. They are:
$$G_{1}=Y_3^2-Y_6X^{\tau} \hspace*{1cm} \mbox{ and} \hspace*{1cm} G_{2}:=Y_6^2-X^{3+2\tau}.$$

\noindent  The unfolds of the above polynomials are
$$\widetilde{G}_{1}=Y_3^2-Y_6X^{\tau}-\displaystyle\sum_{j=1}^{6+8\tau}a_jZ_{6+8\tau-j}  \ \ \mbox{and} \ \
\widetilde{G}_{2}:=Y_6^2-X^{3+2\tau}-\displaystyle\sum_{j=k}^{12+8\tau}b_kZ_{12+8\tau-k},$$

\noindent where the sums vary between the positive integers $j$ and $k$ such that $6+8\tau-j \in \H$ and $12+8\tau-k\in \H$. Doing the variable changes of the form

$$\begin{array}{lcl}
X & \longmapsto & X+\alpha_4\\
Y_3 & \longmapsto & Y_3+\beta_{3+4(\tau-1)} X+\beta_{3+4\tau}\\
Y_6 & \longmapsto & Y_6+\gamma_{3}Y_3+\gamma_{2+4\tau}X+\gamma_{6+4\tau} 
\end{array}$$

\noindent we can normalize $6$ coefficients of the unfolded forms to zero, provided that the characteristic of $\k$ is zero or a odd prime that not divides $\tau$. The unfold of the polynomials $G_1$ and $G_2$ have $3+4\tau$  and $9+3\tau$ coefficients, respectively. Then the parameter space depends on $6+7\tau$ coefficients, i.e. $$\overline{\MS}\simeq\mathbb{P}^{5+7\tau}.$$

\noindent In the particular case $\tau=1,$ we have $\H=<4,7,10>$ and $g=7$. The canonical ideal of the monomial curve $\C^{(0)}$ is generated by $10$ quadratic forms, namely
$$\begin{array}{lll}
F_{8,1}^{(0)}=X_4^2-X_0X_8  & \hspace*{1.5cm}  F_{11,1}^{(0)}= X_4X_7-X_0X_{11} \hspace*{1.5cm} & F_{12,1}^{(0)}= X_4X_8-X_0X_{12}\\
 F_{14,1}^{(0)}=X_7^2-X_4X_{10} &\hspace*{1.5cm} F_{15,1}^{(0)}= X_7X_8-X_4X_{11} \hspace*{1.5cm}&  F_{16,1}^{(0)}=X_8^2-X_4X_{12}\\
 F_{18,1}^{(0)}=X_8X_{10}-X_7X_{11} & \hspace*{1.5cm} F_{19,1}^{(0)}=X_8X_{11}-X_7X_{12}\hspace*{1.5cm} & F_{20,1}^{(0)}=X_{10}^2 -X_8X_{12}\\
& \hspace*{1.5cm} F_{22,1}^{(0)}= X_{11}^2-X_{10}X_{12} \hspace*{1.5cm}&
\end{array}$$
The only syzygy coming from the Syzygy Lemma \ref{syzygylem} is $$X_{12}F^{(0)}_{14,1}-X_{10}F^{(0)}_{16,1}+X_7F^{(0)}_{19,1}-X_8F^{(0)}_{18,1}=0.$$ Applying the shrinking map, i.e. considering this syzygy in $\k[X_4,X_7,X_{10}]$, we have the trivial syzygy $$X_4^3(X_7^2-X_4X_{10})-X_4^3(X_7^2-X_4X_{10}) = X_4^3 F_{14,1}^{(0)}- X_4^3 F_{14,1}^{(0)}=0$$ Thus, as there are no non-trivial syzygies, the space of parameters depends on 13 coefficients, and $$\overline{\MS}\simeq\mathbb{P}^{12}.$$ 
\end{exam}

\begin{exam}\label{exam2}
  Let us now consider an example in codimension 4.  For each $\tau>0$, consider the semigroup
  $\H=<16,1+16\tau,2+16\tau,4+16\tau,8+16\tau>$, whose genus is
  $32\tau$ and Frobenius number $\l_{g}=64\tau-1$. The affine monomial curve
  $$\mathcal{D}=\{(t^{16},t^{1+16\tau},t^{2+16\tau},t^{4+16\tau}, t^{8+16\tau})| t\in\Bbbk\}\subset\mathbb{A}^5 $$ is a complete intersection
  in $\mathbb{A}^4$, its ideal is generated by
  $$G_1=Y_1^2-Y_2X^{\tau} \hspace*{1cm} G_2=Y_2^2-Y_4X^{\tau} \hspace*{1cm}G_3=Y_4^2-Y_8X^{\tau} \hspace*{1cm} G_4=Y_8^2-X^{3\tau}, $$
  where  $X:=X_{16},\ Y_1:=Y_{1+16\tau},\ Y_2:=Y_{2+16\tau}, \ Y_4:=Y_{4+16\tau} \mbox{ and }  Y_5:=Y_{5+16\tau}.$ Unfolding the defining polynomials of $\mathcal{D}$ we get
\begin{equation*}
    \begin{array}{ll}
      \noindent \widetilde{G}_{1}=Y_1^2-Y_2X^{\tau}-\displaystyle\sum_{j=1}^{2+32\tau}a_jZ_{2+32\tau-j}, &
\widetilde{G}_{2}=Y_2^2-Y_4X^{\tau}-\displaystyle\sum_{j=k}^{4+32\tau}b_kZ_{4+32\tau-k}  \\
   \widetilde{G}_{3}=Y_4^2-Y_8X^{\tau}-\displaystyle\sum_{j=u}^{8+32\tau}c_uZ_{8+32\tau-u} \mbox{ and }
 & \widetilde{G}_{4}=Y_8^2-X^{3\tau}-\displaystyle\sum_{j=v}^{16+32\tau}d_vZ_{16+32\tau-v},
    \end{array}
\end{equation*}
with $$\begin{array}{lll}
Z_{16i}:=X^i  & Z_{1+16\tau+8i}=X^iY_1  & Z_{2+16\tau+8i}=X^iY_2  \\ Z_{3+32\tau+8i}=X^iY_1Y_2 &
Z_{4+16\tau+8i}=X^iY_4 & Z_{5+32\tau+8i}=X^iY_1Y_4 \\ Z_{6+32\tau+8i}=X^iY_2Y_4 & Z_{7+48\tau+8i}=X^iY_1Y_2Y_4 &
Z_{8+16\tau+8i}=X^iY_8 \\ Z_{9+32\tau+8i}=X^iY_1Y_8 & Z_{10+32\tau+8i}=X^iY_2Y_8 & Z_{11+48\tau+8i}=X^iY_1Y_2Y_8\\
Z_{12+32\tau+8i}=X^iY_4Y_8 & Z_{13+48\tau+8i}=X^iY_1Y_4Y_8 & Z_{14+48\tau+8i}=X^iY_2Y_4Y_8 \\ Z_{15+64\tau+8i}=X^iY_1Y_2Y_4Y_8
\end{array}$$
We can normalize $15$ coefficients from the unfolding polynomials using 
$$\begin{array}{lcl}
X & \longmapsto & X+\alpha_{16}\\
Y_1 & \longmapsto & Y_1+\beta_{-15+16\tau}X+\beta_{1+16\tau}\\
Y_2 & \longmapsto & Y_2+\gamma_{1}Y_1+\gamma_{-14+16\tau}X+\gamma_{2+16\tau} \\
Y_4 & \longmapsto & Y_4+\theta_2Y_2+d_3Y_1+\theta_{-12+16\tau}X+\theta_{4+16\tau}\\
Y_8 & \longmapsto & Y_8+\mu_4Y_4+\mu_6Y_2+\mu_7Y_1+\mu_{-8+16\tau}X+\mu_{8+16\tau}
\end{array}$$
Hence, counting coefficients we can conclude that $\overline{\MS}\simeq\mathbb{P}^{8+24\tau}$. 
\end{exam}

The GAP System's semigroup package simplifies finding complete intersection numerical semigroups, like $\H=<32,33,34,36,40,48>$ of genus $g=80$. Following the procedure presented here, verifying $\overline{\MS} \simeq \mathbb{P}^{53}$ becomes straightforward. For any such semigroup $\H$, a family like $\H=<32,1+32\tau,2+32\tau,4+32\tau,8+32\tau,16+32\tau>$ ($\tau\geq 1$)  can be considered. Our procedure readily adapts to any family member, as shown in Examples \ref{exam1} and \ref{exam2}.

\bigskip

\parbox[t]{3in}{{\rm Andr\'e Contiero}\\
{\tt \href{mailto:contiero@ufmg.br}{contiero@ufmg.br}}\\
{\it Universidade Federal de Minas Gerais}\\
{\it Belo Horizonte, MG, Brazil}} \hspace{0.9cm}
\parbox[t]{3in}{{\rm Sarah Mazzini}\\
{\tt \href{mailto:sarahmazzini@ufu.br}{sarahmazzini@ufu.br  }}\\
{\it Universidade Federal de Uberl\^andia}\\
{\it Uberl\^andia, MG, Brazil}} \hspace{0.9cm}

\end{document}